\documentclass[11pt]{amsart}

\usepackage{mathpazo}
\usepackage{amsmath,amssymb}
\usepackage{graphicx}
\usepackage{import}
\usepackage{amscd}
\usepackage{wrapfig}
\usepackage{fullpage}
\usepackage{epsfig}
\numberwithin{equation}{section}
\usepackage[alphabetic,nobysame]{amsrefs}
\usepackage{float}
\usepackage{tikz}
\usetikzlibrary{positioning}

\newtheorem{theorem}{Theorem}[section]

\newtheorem{corollary}[theorem]{Corollary}
\newtheorem{proposition}[theorem]{Proposition}

\newcommand{\e}{\epsilon}

\newcommand{\Q}{{\mathcal Q}}
\newcommand{\Qbar}{\overline{\mathcal Q}}
\newcommand{\M}{{\mathcal M}}

\newcommand{\Z}{{\mathbb Z}}
\newcommand{\T}{{\mathcal T}}

\newcommand{\teichmuller}{Teichm{\"u}ller{ }}

\newcommand{\Teich}{{\mathcal T}}

\renewcommand{\ge}{\geqslant}

\makeatletter
 \let\c@theorem=\c@subsection
 \let\c@conjecture=\c@subsection
 \let\c@lemma=\c@subsection
 \let\c@proposition=\c@subsection
 \let\c@claim=\c@subsection
 \let\c@question=\c@subsection
 \let\c@criterion=\c@subsection
 \let\c@vfconj=\c@subsection
 \let\c@definition=\c@subsection
 \let\c@notation=\c@subsection
 \let\c@remark=\c@subsection
 \let\c@example=\c@subsection
 \let\c@equation=\c@subsection
 \let\c@figure=\c@subsection
 \let\c@wrapfigure=\c@subsection

\makeatother

\begin{document}

\title{Recurrence of quadratic differentials for harmonic measure.}

\author[Gadre]{Vaibhav Gadre}
\address{\hskip-\parindent
        School of Mathematics and Statistics\\
        University of Glasgow\\
        University Place\\
        Glasgow G12 8SQ UK}
\email{Vaibhav.Gadre@glasgow.ac.uk}
\thanks{The first author acknowledges support from the GEAR Network
  (U.S. National Science Foundation grants DMS 1107452, 1107263,
  1107367 ``RNMS: GEometric structures And Representation varieties'').}

\author[Maher]{Joseph Maher}
\address{\hskip-\parindent
  Department of Mathematics, College of Staten Island, CUNY \\
  2800 Victory Boulevard, Staten Island, NY 10314, USA \\
  and Department of Mathematics, 4307 Graduate Center, CUNY \\
  365 5th Avenue, New York, NY 10016, USA}
\email{joseph.maher@csi.cuny.edu }
\thanks{The second author acknowledges support from the Simons
  Foundation and PSC-CUNY}

\maketitle

\begin{abstract}
We consider random walks on the mapping class group that have finite
first moment with respect to the word metric, whose support generates
a non-elementary subgroup and contains a pseudo-Anosov map whose
invariant \teichmuller geodesic is in the principal stratum of
quadratic differentials. We show that a \teichmuller geodesic typical
with respect to the harmonic measure for such random walks, is
recurrent to the thick part of the principal stratum. As a
consequence, the vertical and horizontal foliations of such a random
\teichmuller geodesic have no saddle connections.
\end{abstract} 

%%%%%%%%%%%%%%%%%%%%%%%%%%%%%%%%%%%%%%%%%%%%%%%%%%%%%%%%%%%%%%%%%%%%%%%%%%%%%%
\section{Introduction}
%%%%%%%%%%%%%%%%%%%%%%%%%%%%%%%%%%%%%%%%%%%%%%%%%%%%%%%%%%%%%%%%%%%%%%%%%%%%%%

Let $S$ be an orientable surface of finite type.  Let $\text{Mod}(S)$
be the mapping class group of orientation preserving diffeomorphisms
of $S$ modulo isotopy.  Let $\Teich(S)$ be the \teichmuller space of
marked conformal structures on $S$, and the moduli space
$\mathcal{M}(S)$ of Riemann surfaces is the quotient $\Teich(S)/
\text{Mod}(S)$. Let $\mathcal{Q}(S)$ be the space of unit area
quadratic differentials, which may be identified with the unit
cotangent bundle of $\Teich(S)$.  A unit area quadratic differential
determines a unit area flat metric, and we shall only ever consider
flat metrics which have unit area.  We shall write $\pi$ for the
projection map $\pi \colon \Q(S) \to \T(S)$ which sends a quadratic
differential to the underlying Riemann surface.  For punctured
surfaces, the quadratic differentials in $\mathcal{Q}(S)$ are assumed
to be meromorphic with poles only at punctures and with every puncture
a simple pole. The space $\mathcal{Q}(S)$ is stratified by the order
of the zeros of the quadratic differential; the principal stratum
$\Q_{pr}(S)$ consists of those quadratic differentials whose zeros are
all simple.  For the remainder of this paper we shall assume that we
have fixed the surface $S$, and so we shall omit it from our notation,
and just write $\T$ for $\T(S)$, and so on.

The $\epsilon$-thick part of Teichm\"uller space $\T$, which we shall
denote $\T(\e)$, is the collection of all conformal structures
corresponding to hyperbolic metrics in which no simple closed curve
has length less than $\e$.  The complement $\T \setminus \T_\e$ is
called the $\e$-thin part of Teichm\"uller space.  The thick part
$\T(\e)$ is mapping class group invariant, and we shall write $\M(\e)$
for the quotient, which is a subset of moduli space $\M$ and is called
the $\e$-thick part of moduli space.  The $\epsilon$-thick part of a
connected component of a stratum is the set of all quadratic
differentials in the component for which the $q$-length of all saddle
connections is at least $\epsilon$.  The principal stratum is
connected, and we shall write $\Q_{pr}(\e)$ for the $\e$-thick part of
the principal stratum.  Maskit \cite{maskit} showed that the existence
of a short curve in a hyperbolic metric implies that the curve is
short in any compatible unit area flat metric.  More precisely, given
$\e > 0$, there is an $\e' > 0$, such that for any hyperbolic metric
$X$ in the thin part $\T \setminus \T(\e)$, and for any flat metric in
the same conformal class as $X$, the length of any curve in the flat
metric metric is at most $\e'$, i.e. $\pi^{-1}( \T \setminus \T(\e) )
\subset \Q \setminus \Q(\e') $.  A simple closed curve $\alpha$ either
has a unique geodesic representative in the flat metric, which is a
concatenation of saddle connections, or else there is a maximal flat
cylinder on $S$ foliated by parallel closed geodesics.  In the latter
case, the boundary curves of the cylinder will contain singularities,
and hence saddle connections.  In either case, the existence of a
short curve in the flat metric implies the existence of a short saddle
connection.  However, even if there are no short simple closed curves,
there may be arbitrarily short saddle connections.

In summary, the thick part of a strata of quadratic differentials has
a projection into moduli space that is contained in a thick part of
moduli space, i.e. for any $\e > 0$ there is an $\e' > 0$ such that
$\pi( Q_{pr}(\e') ) \subset \T(\e)$.  We remark however, that any
point in $\T$ has a pre-image in $\Q$ which contains points which do
not lie in the thick part of the principal strata.

By the Thurston classification, mapping classes are periodic,
reducible or pseudo-Anosov. A pseudo-Anosov map $g$ has a unique
invariant \teichmuller geodesic $\gamma_g$.  Given a point $X \in
\gamma_g$ there is a unique quadratic differential $q$ at $X$ in the
direction of $\gamma_g$.  If the invariant \teichmuller geodesic is
given by a quadratic differential that lies in the principal stratum,
then we say that the pseudo-Anosov map is in the principal stratum.

We consider random walks on the mapping class group $\text{Mod}(S)$
that have finite first moment with respect to word metric and whose
support generates a non-elementary subgroup of $\text{Mod}(S)$,
i.e. the subgroup generated by the support of the initial distribution
contains a pair of pseudo-Anosov maps with distinct stable and
unstable measured foliations. In independent work, Maher \cite{Mah}
and Rivin \cite{Riv} showed that the probability that a random walk
gives a pseudo-Anosov map tends to $1$ in the length of the sample
path, and in particular, the invariant foliations of pseudo-Anosov
elements do not contain saddle connections.  As a refinement of these
results, we showed the following in \cite{Gad-Mah}, answering a
question of Kapovich and Pfaff \cite{Kap-Pfa}:

\begin{theorem}\label{principal} 
Let $S$ be a connected orientable surface of finite type, whose
\teichmuller space $\Teich(S)$ has complex dimension at least two. Let
$\mu$ be a probability distribution on $\text{Mod}(S)$ such that
\begin{enumerate}
\item $\mu$ has finite first moment with respect to $d_{\text{Mod}}$,
\item $\text{Supp}(\mu)$ generates a non-elementary subgroup $H$ of
$\text{Mod}(S)$, and
\item The semigroup generated by $\text{Supp}(\mu)$ contains a
pseudo-Anosov $g$ such that the invariant \teichmuller geodesic
$\gamma_g$ for $g$ lies in the principal stratum of quadratic
differentials.
\end{enumerate}
Then, for almost every bi-infinite sample path $\omega = (w_n)_{n \in
  \Z}$, there is positive integer $N$ such that for all $n \geqslant
N$ the mapping class $w_n$ is a pseudo-Anosov map in the principal
stratum, that is its invariant \teichmuller geodesic is given by a
quadratic differential with simple zeros and poles. Furthermore,
almost every bi-infinite sample path determines a unique \teichmuller
geodesic $\gamma_\omega$ with the same limit points as the bi-infinite
sample path, and this geodesic also lies in the principal stratum.
\end{theorem}

In this note, we prove the following recurrence result, answering a
further question of Algom-Kfir, Kapovich and Pfaff \cite{akp}:

\begin{theorem} \label{main}%
Let $S$ and $\mu$ satisfy the hypothesis of Theorem
\ref{principal}. Then there exists $\epsilon(S,\mu)> 0$ such that
almost every bi-infinite sample path $\omega = (w_n)_{n \in \Z}$
determines a unique \teichmuller geodesic $\gamma_\omega$ in the
principal stratum of quadratic differentials with the same limit
points in $\text{PMF}(S)$ as $\omega$, and moreover $\gamma_\omega$ is
recurrent to the $\epsilon$-thick part of the principal stratum.
\end{theorem}

Recurrence to the thick part of the moduli space $\mathcal{M}$ is
shown in Kaimanovich-Masur \cite{Kai-Mas} and does not require the
extra hypothesis that the subgroup generated by $\text{Supp}(\mu)$
contains a pseudo-Anosov in the principal stratum. With this extra
hypothesis, Theorem \ref{main} is a finer recurrence statement and
implies their result. A consequence of Theorem \ref{main} and
\cite[Theorem 1]{Mas} is the following refinement of Theorem
\ref{principal}.

\begin{corollary}\label{no-saddles}
Let $S$ and $\mu$ satisfy the hypothesis of Theorem
\ref{principal}. Then almost every bi-infinite sample path $\omega$
determines a unique \teichmuller geodesic $\gamma_\omega$ in the
principal stratum of quadratic differentials with the same limit
points as $\omega$, and the vertical and horizontal projective
measured foliations corresponding to $\gamma_\omega$ are uniquely
ergodic with no vertical and horizontal saddle connections.
\end{corollary}

This corollary follows from the fact that if a quadratic differential
has a saddle connection which is contained in a leaf of the horizontal
or vertical foliations, then the length of this saddle connection
tends to zero in one direction along the geodesic, and so the geodesic
cannot be recurrent to the thick part of a strata.  Corollary
\ref{no-saddles} implies that if one passes from measured foliations
to measured laminations then the lamination given by $\gamma_\omega$
are principal i.e., they have all complementary regions ideal
triangles or once-punctured monogons.

The proof of the recurrence result, Theorem \ref{main}, follows from
the fellow traveling discussion in Section \ref{fellow-thick} below
and the ergodicity of the shift map on $\text{Mod}(S)$.

%%%%%%%%%%%%%%%%%%%%%%%%%%%%%%%%%%%%%%%%%%%%%%%%%%%%%%%%%%%%%%%%%%%%%%%%%%%%%%
\subsection{Acknowledgements}
%%%%%%%%%%%%%%%%%%%%%%%%%%%%%%%%%%%%%%%%%%%%%%%%%%%%%%%%%%%%%%%%%%%%%%%%%%%%%%

We would like to thank Saul Schleimer for helpful conversations.

%%%%%%%%%%%%%%%%%%%%%%%%%%%%%%%%%%%%%%%%%%%%%%%%%%%%%%%%%%%%%%%%%%%%%%%%%%%%%%
\section{Fellow traveling and thickness}\label{fellow-thick}
%%%%%%%%%%%%%%%%%%%%%%%%%%%%%%%%%%%%%%%%%%%%%%%%%%%%%%%%%%%%%%%%%%%%%%%%%%%%%%

Let $\mathcal{Q}_{\text{pr}}$ be the principal stratum of quadratic
differentials. Let $\mathcal{Q}_{\text{pr}}(\epsilon)$ be the set of
principal quadratic differentials $q$ for which every saddle
connection $\beta$ on $q$ satisfies $\ell_q(\beta) \ge \epsilon$ in
the induced unit area flat metric on $S$.  We shall write $\Qbar_{pr}$
for the quotient of $\Q_{pr}$ by the mapping class group.

A quadratic differential $q$ determines a Teichmuller geodesic
$\gamma$ in $\T$, and we shall write $\widetilde \gamma$ for the
corresponding image of $q$ in $Q$ under the geodesic flow, which
projects down to $\gamma$.  Given a quadratic differential $q$, we
shall parameterize the corresponding geodesic by setting $q(0) = q$
and $\gamma(0) = \pi( q(0) )$.  We shall write $\gamma_t$ for the
point in $\T$ distance $t$ along the geodesic in $\T$, and $q(t)$ for
the corresponding point in $\widetilde \gamma$, so $\gamma(t) = \pi(
q(t) )$.

We say a Teichm\"uller geodesic $\gamma$ is \emph{recurrent in $\M$ in
  the forward direction} if there is a compact set $K$ in $\M$, and a
sequence of points $t_n \to \infty$, such that $\gamma({t_n}) \in K$.
For any compact set in $\M$ there is an $\e > 0$ such that $K$ is
contained in the $\e$-thick part of $\M$, so recurrent in $\M$ implies
recurrence to $\M(\e)$ for some $\e > 0$.  Masur \cite{Mas} showed
that if $\gamma$ is recurrent in $\M$, then $\gamma$ has a uniquely
ergodic vertical foliation.  We say a Teichm\"uller geodesic $\gamma$
is \emph{recurrent in $\Qbar_{pr}$ in the forward direction} if there
is a compact set $K$ in $\Qbar_{pr}$, and a sequence of points $t_n
\to \infty$, such that $q(t_n) \in K$.  Any compact set in
$\Qbar_{pr}$ is contained in $\Qbar_{pr}(\e)$ for some $\e > 0$, so
recurrence in $\Qbar_{pr}$ implies recurrence to the thick part
$\Qbar_{pr}(\e)$, for some sufficiently small $\e$.  Recurrence in
$\Qbar_{pr}$ implies recurrence in $\M$, and furthermore recurrence in
$\Qbar_{pr}$ implies that the vertical foliation of $\gamma$ contains
no saddle connections, as the length of a vertical saddle connection
tends to zero as $t \to \infty$.

\begin{proposition}\label{subseq}
Suppose that a \teichmuller geodesic $\gamma$, determined by a
quadratic differential $q_0 \in \Q_{pr}(\e)$, is recurrent to
$\Qbar_{pr}(\e)$ in both the forwards and backwards directions.
Suppose $\tau_n$ is sequence of \teichmuller geodesic segments that
$R$-fellow travel $\gamma$ for distance $d_n$ such that the midpoints
$X_n$ of $\tau_n$ are within \teichmuller distance $R$ of $X_0$ and
$d_n \to \infty$.  Let $q_n$ be the quadratic differential at $X_n$
corresponding to $\gamma_n$.  Then there exists $\e' > 0$, depending
on $q_0$ and $R$, and a subsequence $n_k$ with $k \in \mathbb{Z}$ such
that $q_{n_k} \in \mathcal{Q}_{\text{pr}}(\e')$ as $k \to \pm \infty$.
\end{proposition}

\begin{proof}
As the Teichm\"uller geodesic $\gamma$ is recurrent to the thick part
$\Qbar_{pr}(\e)$, it is also recurrent to a thick part of $\M(\e_1)$
for some $\e_1 > 0$.  By work of Masur \cite{Mas}, as the
Teichm\"uller geodesic is recurrent in both directions, this implies
that both the vertical and horizontal foliations are uniquely ergodic.
As $\Qbar_{pr}$ is open, we may choose an open neighbourhood $U$ of
$\{ q_t \mid t \in (-R, R) \}$ in $\Qbar_{pr}$ which is contained in
$\Qbar_{pr}$, and whose closure $K = \overline{U}$ is also contained
in $\Qbar_{pr}$, and is compact.  In particular, there is an $\e_2 >
0$ such that $K \subset \Qbar_{pr}(\e_2)$.

By convergence on compact sets, one can pass to a subsequence of
$\tau_n$'s that converges to bi-infinite \teichmuller geodesic
$\gamma'$ whose vertical and horizontal foliations have intersection
number zero with the vertical and horizontal foliations $(F_s, F_u)$
of $\gamma$.  Hence, the vertical and horizontal foliations of
$\gamma'$ are also $F_s$ and $F_u$. Since a \teichmuller geodesic with
this foliation data has to be unique, $\gamma' = \gamma$. In
particular, by passing to a subsequence we get that $q_{n_k} \to
q_0(s)$ for some $s \in (-R, R)$.  So the tail of the sequence
$q_{n_k}$ must consists of quadratic differentials in $K \subset
\Qbar_{\text{pr}}$ and moreover in $\Qbar_{\text{pr}}(\e')$ as $k \to
\infty$ proving the proposition.
\end{proof}

Let $g$ be a pseudo-Anosov map whose invariant \teichmuller geodesic
$\gamma_g$ is in the principal stratum. Also suppose that $\epsilon$
has been chosen small enough such that $\gamma_g$ is contained in
$\mathcal{Q}_{\text{pr}} (\epsilon)$.

\begin{proposition}
Given a pseudo-Anosov element $g$ and a constant $R$, there is an $\e
> 0$, such that if $\gamma$ is a \teichmuller geodesic which has
sequences $T_n, d_n$ for $n \in \mathbb{N}$ such that
\begin{enumerate}
\item $T_n, d_n \to \infty$ as $n \to \infty$, and
\item there are mapping classes $h_n$ such that the geodesic
$\gamma_n= h_n (\gamma_g)$ has a segment that $R$-fellow travels
$\gamma_t$ over the time interval $(T_n - d_n, T_n+d_n)$.
% \item the quadratic differential along $\gamma_n$ at the midpoint of
% the fellow travelling segment is in
% $\mathcal{Q}_{\text{pr}}(\epsilon)$.
\end{enumerate}
Then there is a subsequence $n_k$ such that $q_{T_{n_k}} \in
\mathcal{Q}_{\text{pr}}(\e)$.
\end{proposition}

\begin{proof}
Pulling back by $h_n^{-1}$, the sequence of geodesic segments $g_n =
h_n^{-1}(\gamma(T_n - d_n, T_n +d_n))$ satisfy the hypothesis of
Proposition \ref{subseq} with respect to the geodesic $\gamma_g$ which
is recurrent by the virtue of being thick. The proposition then
follows from Proposition \ref{subseq}.
\end{proof}

%%%%%%%%%%%%%%%%%%%%%%%%%%%%%%%%%%%%%%%%%%%%%%%%%%%%%%%%%%%%%%%%%%%%%%%%%%%%%%
\section{Random walks and recurrence}
%%%%%%%%%%%%%%%%%%%%%%%%%%%%%%%%%%%%%%%%%%%%%%%%%%%%%%%%%%%%%%%%%%%%%%%%%%%%%%

We recall some terminology and results from \cite{Gad-Mah}. For a
point $X \in \Teich(S)$ and $r> 0$ let $B_r(X)$ be the ball of radius
$r$ centred at $X$. Let $\gamma$ be a \teichmuller geodesic. For
points $X$ and $Y$ on $\gamma$ let $\Gamma_r(X, Y)$ be the set of
\teichmuller geodesics that pass through $B_r(X)$ and $B_r(Y)$.  By
work of Rafi \cite{Raf}, if $X$ and $Y$ lie in the thick part
$\M(\e)$, then there is an $R$, that depends on $r$ and $\e$, such
that every geodesic in $\Gamma_r(X,Y)$ fellow travels with constant
$R$ the geodesic segment $[X,Y]$ of $\gamma$.

Now let $g$ be a pseudo-Anosov element in $\text{Supp}(\mu)$ such that
$\mu^{(j)}(g) > 0$ for some $j \in \mathbb{N}$ and the invariant
\teichmuller geodesic $\gamma_g$ is in the principal stratum of
quadratic differentials. Without loss of generality, we choose a
base-point $X$ on $\gamma_g$. Following the proof of \cite[Theorem
1.1]{Gad-Mah}, for all $k \in \mathbb{N}$ large enough let $\Omega_k$
be the set of bi-infinite sample paths $\omega= (w_n)_{n \in
  \mathbb{Z}}$ such that the sequence $w_n X$ converges to uniquely
ergodic foliations $F_+$ and $F_-$ as $n \to \infty$ and $n \to -
\infty$ respectively and the \teichmuller geodesic $\gamma(F_-, F_+)$
is contained in $\Gamma_r(g^{-k}X, g^kX)$.

Let $\nu$ be the harmonic measure and $\hat{\nu}$ be the reflected
harmonic measure. Let $\sigma: \text{Mod}^{\mathbb{Z}} \to
\text{Mod}^{\mathbb{Z}}$ be the shift map. Following the proof of
\cite[Theorem 1.1]{Gad-Mah}, we get the following result

\begin{proposition}
Let $S$ and $\mu$ satisfy the hypothesis of Theorem \ref{principal}.
For any large $k$ and for almost every bi-infinite sample path
$\omega$, there is a sequence of times $n_j \to \infty$ as $j \to
\infty$ such that $\sigma^{n_j}(\omega) \in \Omega_k$.
\end{proposition}

Since a countable intersection of full measure sets has full measure
we get that
\begin{proposition} \label{fellow} %
Let $S$ and $\mu$ satisfy the hypothesis of Theorem \ref{principal}.
For almost every bi-infinite sample path $\omega$ there is a sequence
$m_k \to \infty$ as $k \to \infty$ such that $\sigma^{m_k}(\omega) \in
\Omega_k$ for all $k$ large enough.
\end{proposition}

Now we get to the proof of the main recurrence result, Theorem \ref{main}:

\begin{proof}[Proof of Theorem \ref{main}]
By Proposition \ref{fellow}, for almost every sample path $\omega =
(w_n)$ there exists a sequence $m_k$ such that $\gamma_\omega$ fellow
travels $w_{m_k} (\gamma_g)$ between $w_{m_k} g^{-k} X$ and $w_{m_k}
g^k X$. Equivalently, the geodesics $w_{m_k}^{-1} (\gamma_\omega)$
fellow travels $\gamma_g$ between $[g^{-k}X, g^k X]$. The distances
$d_{\mathcal{T}}(g^{-k}X, g^k X)$ form a sequence that tends to
infinity as $k \to \infty$. So by Proposition \ref{subseq}, a further
subsequence of quadratic differentials given by the midpoints of the
fellow travelling segments of $w_{m_k}^{-1} (\gamma_\omega)$ are in
$\Qbar_{\text{pr}}(\e)$. Thus $\gamma_\omega$ is recurrent to
$\Qbar_{\text{pr}}(\e)$.
\end{proof}

\begin{proof}[Proof of Corollary \ref{no-saddles}]
By Theorem \ref{main}, for almost every sample path $\omega$ the
tracked \teichmuller geodesic $\gamma_\omega$ is recurrent to the
thick part $\Qbar_{\text{pr}}(\e)$. The projection to moduli space
$\M$ of $\gamma_\omega$ is then recurrent to the thick part $\M(\e')$
for some $\e' > 0$. By Masur's theorem \cite{Mas}, the vertical
foliation $F_s$ of $\gamma_t$ is uniquely ergodic. Moreover,
recurrence to $\Qbar_{pr}(\e)$ implies that $F_s$ has no vertical
saddle connections
\end{proof}

\end{document}